\documentclass[11pt,reqno]{amsart}
\usepackage{amssymb,latexsym,mathtools}
\usepackage{graphicx}
\usepackage{url}		

\calclayout

\makeatletter
\newcommand{\monthyear}[1]{%
\def\@monthyear{\uppercase{#1}}}
\newcommand{\volnumber}[1]{%
\def\@volnumber{\uppercase{#1}}}
\AtBeginDocument{%
    \def\ps@plain{\ps@empty
        \def\@oddfoot{\@monthyear \hfil \thepage}%
    \def\@evenfoot{\thepage \hfil \@volnumber}}
    \def\ps@firstpage{\ps@plain}
    \def\ps@headings{\ps@empty
        \def\@evenhead{%
            \setTrue{runhead}%
            \def\thanks{\protect\thanks@warning}%
        \uppercase{}\hfil}%
        \def\@oddhead{%
            \setTrue{runhead}%
            \def\thanks{\protect\thanks@warning}%
        \hfill\uppercase{$S$-Legal Index Difference Sequences}}%
        \let\@mkboth\markboth
        \def\@evenfoot{%
        \thepage \hfil \@volnumber}%
        \def\@oddfoot{%
        \@monthyear \hfil \thepage}%
    }%
    \footskip=25pt
    \pagestyle{headings}%
}
\makeatother

\renewcommand{\ge}{\geqslant}
\renewcommand{\geq}{\geqslant}
\renewcommand{\le}{\leqslant}
\renewcommand{\leq}{\leqslant}

\theoremstyle{plain}
\numberwithin{equation}{section}
\newtheorem{thm}{Theorem}[section]
\newtheorem{theorem}[thm]{Theorem}
\newtheorem{lemma}[thm]{Lemma}
\newtheorem{corollary}[thm]{Corollary}
\newtheorem{example}[thm]{Example}
\newtheorem{definition}[thm]{Definition}
\newtheorem{proposition}[thm]{Proposition}

\usepackage{tikz}
\usepackage[colorlinks]{hyperref}
\hypersetup{
    linkcolor=blue!50!gray,
    citecolor=blue!50!gray,
    filecolor=blue!50!gray,
    urlcolor =blue!50!gray,
    menucolor=blue!50!gray,
    runcolor =blue!50!gray
}
\newcommand{\emphdefn}[1]{\textbf{#1}}

\newtheorem{conjecture}[thm]{Conjecture}
\theoremstyle{definition}
\newtheorem{question}{Question}

\begin{document}
\monthyear{}
\volnumber{}
\setcounter{page}{1}

\title{Recurrence relations for $S$-legal index difference sequences}

\thanks{This research was conducted as part of the \textbf{2022 SMALL REU program} at Williams College, and was funded by \textbf{NSF Grant DMS1947438} and \textbf{Williams College} funds. We thank our colleagues from the 2022 SMALL REU program for many helpful conversations.}

\author{Guilherme~Zeus Dantas~e~Moura}
\address{Department of Mathematics and Statistics\\
    Haverford College\\
    Haverford, PA 19041, USA}
\email{zeusdanmou@gmail.com}

\author{Andrew Keisling}
\address{Department of Mathematics\\
    University of Michigan\\
Ann Arbor, MI 48109, USA}
\email{keislina@umich.edu}

\author{Astrid Lilly}
\address{Mathematics Department\\
    Reed College\\
    Portland, OR 97202, USA}
\email{astridlilly@reed.edu}

\author{Annika Mauro}
\address{Department of Mathematics\\
    Stanford University\\
    Stanford, CA 94305, USA}
\email{amauro@stanford.edu}

\author{Steven J. Miller}
\address{Department of Mathematics and Statistics\\
    Williams College\\
    Williamstown, MA 01267, USA}
\email{sjm1@williams.edu}

\author{Matthew Phang}
\address{Department of Mathematics and Statistics\\
    Williams College\\
    Williamstown, MA 01267, USA}
\email{mlp6@williams.edu}

\author{Santiago Velazquez Iannuzzelli}
\address{Department of Mathematics\\
    University of Pennsylvania\\
    Philadelphia, PA 19104, USA}
\email{smvelian@sas.upenn.edu}

\begin{abstract}
    Zeckendorf's Theorem implies that the Fibonacci number $F_n$ is the smallest positive integer that cannot be written as a sum of non-consecutive previous Fibonacci numbers. Catral et al.\ studied a variation of the Fibonacci sequence, the Fibonacci Quilt sequence: the plane is tiled using the Fibonacci spiral, and integers are assigned to the squares of the spiral such that each square contains the smallest positive integer that cannot be expressed as the sum of non-adjacent previous terms. This adjacency is essentially captured in the differences of the indices of each square: the $i$\textsuperscript{th} and $j$\textsuperscript{th} squares are adjacent if and only if $|i - j| \in \{1, 3, 4\}$ or $\{i, j\} = \{1, 3\}$.

    We consider a generalization of this construction: given a set of positive integers $S$, the $S$-legal index difference ($S$-LID) sequence $(a_n)_{n=1}^\infty$ is defined by letting $a_n$ to be the smallest positive integer that cannot be written as $\sum_{\ell \in L} a_\ell$ for some set $L \subset [n-1]$ with $|i - j| \notin S$ for all $i, j \in L$. We discuss our results governing the growth of $S$-LID sequences, as well as results proving that many families of sets $S$ yield $S$-LID sequences which follow simple recurrence relations.
\end{abstract}

\maketitle


\section{Introduction}

\subsection{Fibonacci numbers and Zeckendorf's Theorem}

\begin{definition}[Fibonacci sequence] \label{definition:fibonacci}
    The \emphdefn{Fibonacci sequence} $F_1, F_2, \dots$ is defined by $F_1 = 1$, $F_2 = 2$, and $F_{n} = F_{n-1} + F_{n-2}$ for all integers $n \geq 3$.
\end{definition}

Note that the indices are shifted from the standard definition of the Fibonacci sequence.
With this non-standard notation, the statement of Zeckendorf's Theorem is very beautiful.

\begin{theorem}[Zeckendorf's theorem, \cite{ZeckendorfsTheorem}] \label{theorem:zeckendorf}
    Every nonnegative integer is uniquely decomposed as a sum of distinct Fibonacci numbers without consecutive indices.
\end{theorem}

A decomposition of a nonnegative integer $m$ into a sum of distinct Fibonacci numbers without consecutive indices is called a \emphdefn{Zeckendorf decomposition} of $m$.

For example, the Zeckendorf decomposition of $2022$ is   
\begin{equation}
    2022\ =\ F_{16} + F_{13} + F_{8} + F_{6} + F_{1}.
\end{equation}

Zeckendorf's~theorem yields a well-known alternate description of the Fibonacci sequence, as in Proposition~\ref{proposition:alternate-fibonacci}.

\begin{proposition} \label{proposition:alternate-fibonacci}
    For each positive integer $n$, $F_n$ is the smallest positive integer which cannot be decomposed as a sum of distinct numbers in $\{F_1, F_2, \dots, F_{n-1}\}$ without consecutive indices.
\end{proposition}

Note that Proposition~\ref{proposition:alternate-fibonacci} uniquely describes the Fibonacci numbers. For example, using Proposition~\ref{proposition:alternate-fibonacci}, we assert that
\begin{itemize}
    \item $F_1 = 1$ since no positive integer can be decomposed as a sum of distinct numbers in the empty set without consecutive indices;
    \item $F_2 = 2$ since the only positive integer that can be decomposed as a sum of distinct numbers in $\{F_1\} = \{1\}$ without consecutive indices is $1$;
    \item $F_3 = 3$ since the positive integers that can be decomposed as a sum of distinct numbers in $\{F_1, F_2\} = \{1, 2\}$ without consecutive indices are $1$ and $2$; and
    \item $F_4 = 5$ since the positive integers that can be decomposed as a sum of distinct numbers in $\{F_1, F_2, F_3\} = \{1, 2, 3\}$ without consecutive indices are $1$, $2$, $3$, and $4 = 1 + 3$.
\end{itemize}

\subsection{Generalizations of the Fibonacci sequence}

Several authors have defined generalizations of the Fibonacci sequence by changing some parameters in the interpretation of the Fibonacci sequence described in Proposition~\ref{proposition:alternate-fibonacci}. 
The common theme in many of such generalizations is changing what is an allowed decomposition.


In \cite{Ostrowski}, Ostrowski showed that non-negative integers can be decomposed as integer linear combinations of terms in the continuant sequence of an irrational number $\alpha$, and that this decomposition is unique given some additional constraints. Zeckendorf's theorem can then be viewed as the special case of this result with $\alpha=\varphi$.
Daykin \cite{Daykin} investigated properties of sequences of positive integers that admit decompositions of every integer.
In \cite{LegalDecompositionsNPLRS}, Catral, Ford, Harris, Nelson, and the fifth author defined the $(s,b)$-Generacci sequence, in which the terms of the sequence are put into bins of size $b$, and legal decompositions of integers can contain at-most one term from each bin. Furthermore, the distance between bins from which numbers are taken must greater than $s$. With this notation, the Fibonacci sequence is the $(1,1)$-Generacci sequence, and the authors gave a closed-form expression for the terms of the $(1,2)$-Generacci sequence.
In a different paper \cite{f-decompositions}, Demontigny, Do, Kulkarni, Moon, Varma, and the fifth author studied yet another generalization in which the notion of a legal decomposition varies with index based on some function $f$.
For other examples of related work in this subject, see \cite{Alpert},
\cite{Grabner}, \cite{Filipponi}, \cite{Timothy}, and \cite{Hamlin}.

\subsection{Quilt sequences}

Let's visualize Proposition~\ref{proposition:alternate-fibonacci}.
We arrange a sequence of boxes in an array that extends infinitely to the right, as seen in as seen in Figure~\ref{figure:fibonacci-arrangement}. This construction is called the \emphdefn{Fibonacci array of boxes}.
We assign the Fibonacci number $F_n$ with the $n$\textsuperscript{th} box.
We rewrite Proposition~\ref{proposition:alternate-fibonacci} as Proposition~\ref{proposition:visual-fibonacci}.
\begin{proposition}\label{proposition:visual-fibonacci}
    For each positive integer $n$, $F_n$ is the smallest positive integer that cannot be written as a sum of distinct numbers in $\{F_1, F_2, \dots, F_{n-1}\}$ without using two summands for which the correspondent boxes in the Fibonacci array of boxes (Figure~\ref{figure:fibonacci-arrangement}) share an edge.
\end{proposition}

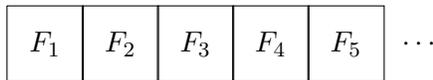
\begin{figure}[htbp]
    \centering
    \begin{tikzpicture}
    \foreach \x [count = \xx] in {$F_1$,$F_2$,$F_3$,$F_4$,$F_5$,$\cdots$}{
        \def\mdot{$\cdots$}
        \node[inner sep=8pt, minimum size=1cm] (char) at (\xx*1, 0) {\x};
        \ifx\x\mdot
        \else
            \draw (char.north west) -- (char.south west) -- (char.south east) -- (char.north east) -- (char.north west) ;
        \fi
    }
\end{tikzpicture}
    \caption{The (start of the) Fibonacci array of boxes.}
    \label{figure:fibonacci-arrangement}
\end{figure}

Inspired by Proposition~\ref{proposition:visual-fibonacci}, Catral, Ford, Harris, Nelson, and the fifth author \cite{LegalDecompositionsNPLRS} defined the Fibonacci quilt sequence by changing the 1-dimensional process of the Fibonacci array of boxes to a 2-dimensional process. The 2-dimensional process that they use is the Fibonacci spiral (Figure~\ref{figure:fibonacci-spiral}).
With respect to adjacency, the Fibonacci spiral can be viewed as the log cabin quilt pattern (Figure~\ref{figure:fibonacci-quilt-pattern}), i.e., the $i$\textsuperscript{th} and $j$\textsuperscript{th} boxes in the Fibonacci spiral are adjacent if, and only if, $i$\textsuperscript{th} and $j$\textsuperscript{th} boxes in the log cabin quilt pattern are adjacent.
This change is purely aesthetical, since the drawing of the log cabin quilt pattern is more compact to be drawn.

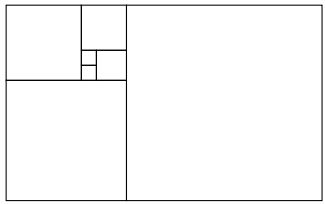
\begin{figure}[htbp]
    \centering
    \begin{tikzpicture}[scale = .2]
    \draw (0, 0)   rectangle (-1, -1);
    \draw (-1, -1) rectangle (0, -2);
    \draw (0, -2)  rectangle (2, 0);
    \draw (2, 0)   rectangle (-1, 3);
    \draw (-1, 3)  rectangle (-6, -2);
    \draw (-6, -2) rectangle (2, -10);
    \draw (2, -10) rectangle (15, 3);
\end{tikzpicture}
    \caption{The (start of the) Fibonacci spiral.}
    \label{figure:fibonacci-spiral}
\end{figure}

\begin{figure}[htbp]
    \centering
    \begin{tikzpicture}[scale=.3]
    \draw (0,1) -- (1,1) -- (1,2) -- (0,2) -- (0,1);
    \draw (0,0) -- (1,0) -- (1,1) -- (0,1) -- (0,0);
    \foreach \x in {2,...,6}
    \draw (\x,-\x+2) -- (\x,\x) -- (\x-1,\x) -- (\x-1,-\x+2) -- (\x,-\x+2);
    \foreach \x in {2,...,6}
    \draw (-\x+2,-\x+2) -- (-\x+2,\x+1) -- (-\x+1,\x+1) -- (-\x+1,-\x+2) -- (-\x+2,-\x+2);
    \foreach \x in {1,...,5}
    \draw (-\x+1,\x+2) -- (-\x+1,\x+1) -- (\x+1,\x+1) -- (\x+1,\x+2) -- (-\x+1,\x+2);
    \foreach \x in {2,...,6}
    \draw (-\x+1,-\x+2) -- (-\x+1,-\x+1) -- (\x,-\x+1) -- (\x,-\x+2) -- (-\x+1,-\x+2);
    \node at (.5,.5) {};
    \node at (.5,1.5) {};
    \foreach \x in {3,7,...,20}
    \node at (\x/4+.75,1) {};
    \foreach \x in {5,9,...,22}
    \node at (-\x/4+.75,1.5) {};
    \foreach \x in {4,8,...,21}
    \node at (1,\x/4+1.5) {};
    \foreach \x in {6,10,...,19}
    \node at (.5,-\x/4+1) {};
    \node at (.5,-26/4+2) {}; 
\end{tikzpicture}

    \caption{The (start of the) log cabin quilt pattern.}
    \label{figure:fibonacci-quilt-pattern}
\end{figure}
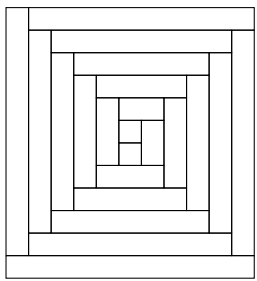

\begin{definition}[Fibonacci quilt sequence] \label{definition:fibonacci-quilt-sequence}
    The \emphdefn{Fibonacci quilt sequence} is the sequence $q_1, q_2, \dots$ of positive integers such that, for each positive integer $n$, $q_n$ is the smallest positive integer that cannot be written as a sum of distinct numbers in $\{q_1, q_2, \dots, q_{n-1}\}$ without using two summands for which the correspondent boxes in the log quilt pattern (Figure~\ref{figure:fibonacci-quilt-pattern}) share an edge.
\end{definition}

The first ten terms of the Fibonacci quilt sequence are $1, 2, 3, 4, 5, 7, 9, 12, 16, 21$.
Except on initial terms, the Fibonacci quilt sequence agrees with the Padovan sequence, OEIS sequence A000931 \cite{OEISA000931}. In particular, the Fibonacci quilt sequence satisfies simple recurrence relations, shown in Theorem~\ref{theorem:fibonacci-quilt-recurrence}.

\begin{theorem}[Recurrence relations, \cite{LegalDecompositionsNPLRS}] \label{theorem:fibonacci-quilt-recurrence}
    Let $q_n$ denote the $n$\textsuperscript{th} term in the Fibonacci quilt sequence. Then
    \begin{align}
        \text{for } n \geq 6,& \quad q_{n+1} \ =\ q_{n} + q_{n-4}, \label{equation:fibonacci-quilt-recurrence} \\ 
        \text{for } n \geq 5,& \quad q_{n+1} \ =\ q_{n-1} + q_{n-2}.
    \end{align}
\end{theorem}

We define a similar sequence that arises from another 2-dimensional process: the Padovan triangle spiral, depicted in Figure~\ref{figure:padovan-spiral}. The Padovan triangle spiral is also called the triangular quilt.

\begin{figure}[htbp]
    \centering
    \begin{tikzpicture}
	\coordinate (A1) at (0, 0);
	\coordinate[shift={(.6, 0)}] (A2) at (A1);
	\coordinate[rotate around={60:(A1)}] (A3) at (A2);
	\coordinate[rotate around={60:(A1)}] (A4) at (A3);
	\coordinate[rotate around={60:(A1)}] (A5) at (A4);
	\coordinate[rotate around={60:(A2)}] (A6) at (A5);
	\coordinate[rotate around={60:(A2)}] (A7) at (A6);
	\coordinate[rotate around={60:(A3)}] (A8) at (A7);
	\coordinate[rotate around={60:(A4)}] (A9) at (A8);
	\coordinate[rotate around={60:(A5)}] (A10) at (A9);
	\coordinate[rotate around={60:(A6)}] (A11) at (A10);

	\draw (A1) -- (A2) -- (A3) -- cycle;
        \coordinate (T1) at (barycentric cs:A1=1,A2=1,A3=1);
	\draw (A1) -- (A3) -- (A4) -- cycle;
        \coordinate (T2) at (barycentric cs:A1=1,A3=1,A4=1);
	\draw (A1) -- (A4) -- (A5) -- cycle;
        \coordinate (T3) at (barycentric cs:A1=1,A4=1,A5=1);
	\draw (A2) -- (A5) -- (A6) -- cycle;
        \coordinate (T4) at (barycentric cs:A2=1,A5=1,A6=1);
	\draw (A2) -- (A6) -- (A7) -- cycle;
	\coordinate (T5) at (barycentric cs:A2=1,A6=1,A7=1);
	\draw (A3) -- (A7) -- (A8) -- cycle;
	\coordinate (T6) at (barycentric cs:A3=1,A7=1,A8=1);
	\draw (A4) -- (A8) -- (A9) -- cycle;
	\coordinate (T7) at (barycentric cs:A4=1,A8=1,A9=1);
	\draw (A5) -- (A9) -- (A10) -- cycle;
	\coordinate (T8) at (barycentric cs:A5=1,A9=1,A10=1);
	\draw (A6) -- (A10) -- (A11) -- cycle;
	\coordinate (T9) at (barycentric cs:A6=1,A10=1,A11=1);
\end{tikzpicture}
    \caption{The (start of the) Padovan triangle spiral, or triangular quilt.}
    \label{figure:padovan-spiral}
\end{figure}
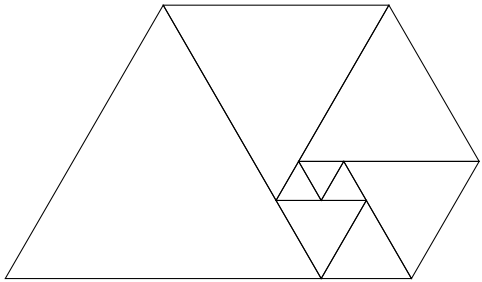

\begin{definition}[Triangular quilt sequence] \label{definition:triangular-quilt-sequence}
    The \emphdefn{triangular quilt sequence} is the sequence $t_1, t_2, \dots$ of positive integers such that, for each positive integer $n$, $t_n$ is the smallest positive integer that cannot be written as a sum of distinct numbers in $\{t_1, t_2, \dots, t_{n-1}\}$ without using two summands for which the correspondent triangles the triangular quilt (Figure~\ref{figure:padovan-spiral}) share an edge.
\end{definition}

The first few terms of the triangular quilt sequence are%
\footnote{Refer to the GitHub repository at \url{https://github.com/ZeusDM/S-LID-sequences} for the algorithm we used to generate these terms.}
\begin{equation}
    1, 2, 3, 5, 6, 11, 12, 20, 23, 40, 46, 80, 92, 152, 175, 295, 341, 573, 665, 1164, 1339, 2219, 2560.
\end{equation}

Since the definition of the triangular quilt sequence is very similar to definition of the Fibonacci quilt sequence, we expect to find a recurrence relation for the triangular quilt sequence, similar to Theorem~\ref{theorem:fibonacci-quilt-recurrence}.
However, further calculations fail to reveal such a recurrence.
Nevertheless, Proposition~\ref{proposition:triangular-quilt-recurrence} displays a natural generalization of Equation~\eqref{equation:fibonacci-quilt-recurrence} that mysteriously holds for some, but not all, values of $n$.

\begin{proposition} \label{proposition:triangular-quilt-recurrence}
    Let $6 \le n \le 50$ be an integer.
    Then,
    \begin{equation} \label{equation:recurrence-triangular}
        t_{n+1}\ =\ t_{n} + t_{n-5}
    \end{equation}
    if, and only if, $n \in \{6, 8, 10, 12, 14, 16, 18, 20, 22, 24, 26, 31, 33, 35, 37, 39, 41, 43, 46, 48\}$.
\end{proposition}

Proposition~\ref{proposition:triangular-quilt-recurrence} was verified using computer code.%
\footnote{Refer to the GitHub repository at \url{https://github.com/ZeusDM/S-LID-sequences} for the details of this verification.} We conjecture that Equation~\eqref{equation:recurrence-triangular} holds for infinitely many indices $n$.


\subsection{\texorpdfstring{$S$}{S}-legal index difference sequences} \label{subsection:slid-sequences}

Intrigued by the apparent lack of a simple recurrence relation for the triangular quilt sequence, we decided to simplify and generalize the definition of the quilt sequences.

We observe that the boxes corresponding to $q_i$ and $q_j$ share an edge in the log cabin quilt pattern (Figure~\ref{figure:fibonacci-quilt-pattern}) if, and only if, $|i - j| \in \{1, 3, 4\}$ or $\{i, j\} = \{1, 3\}$.
Similarly, the triangles corresponding to $t_i$ and $t_j$ share an edge in the Padovan triangle spiral (Figure~\ref{figure:padovan-spiral}) if, and only if, $|i-j| \in \{1, 5\}$ or $\{i, j\} = \{1, 4\}$.

In Definition~\ref{definition:slid-sequence}, we create the $S$-legal index difference sequences, $S$-LID sequence for short.
We disregard the initial edge cases of the adjacency of the log cabin quilt pattern and the Padovan triangle spiral (respectively, $\{i, j\} = \{1, 3\}$ and $\{i, j\} = \{1, 4\}$) and focus on the fact that adjacency means that the index difference is in a fixed set (respectively, the sets $\{1, 3, 4\}$ and $\{1, 5\}$).

\begin{definition}[$S$-legal index difference sequences] \label{definition:slid-sequence}
    Let $S$ be a set of positive integers.
    The \emphdefn{$S$-legal index difference sequence} (or $S$-LID sequence) is the sequence $a_1, a_2, \dots$ of positive integers such that, for each positive integer $n$, $a_n$ is the smallest positive integer that cannot be written as a sum of distinct numbers in $\{a_1, a_2, \dots, a_{n-1}\}$ without using two summands $a_i$ and $a_j$ such that $|i - j| \in S$.
\end{definition}

Note that for any set $S$ of positive integers, the first and second terms of the $S$-LID sequence are $1$ and $2$.
Also, note that the $S$-LID sequence is increasing.

\begin{definition}[$S$-legal index difference decomposition] \label{definition:slid-decomposition}
    Let $S$ be a finite set of positive integers.
    Let $a_1, a_2, \dots$ be the $S$-LID sequence.
    A \emphdefn{$S$-legal index difference decomposition} (or $S$-LID decomposition) of a nonnegative integer $m$ is a sum of the form
    \begin{equation}
        m\ =\ \sum_{\ell \in L} a_{\ell},
    \end{equation}
    in which $L \subset \mathbb{Z}_{>0}$ and $|i - j| \notin S$ for all $i, j \in L$.
\end{definition}

The definitions~\ref{definition:slid-sequence} and~\ref{definition:slid-decomposition} imply that, for every subset $S \subset \mathbb{Z}_{>0}$, each nonnegative integer has at least one $S$-LID decomposition.
Also, the $n$\textsuperscript{th} term of the $S$-LID sequence is the smallest integer that does not have a $S$-LID decomposition using the first $n-1$ terms of the $S$-LID sequence.

Refer to Section~\ref{section:examples} in order to see examples of $S$-LID sequences.
In addition to examples, we believe that Proposition~\ref{proposition:lower-bound-relation} is very useful in understanding the $S$-LID sequences.

\begin{proposition}[Fundamental lower bound of $S$-LID sequences] \label{proposition:lower-bound-relation}
    Let $S$ be a finite non-empty set of positive integers. Let $k = \max S$.
    Let $a_1, a_2, \dots$ be the $S$-LID sequence.
    Then, for all $n \geq k+1$,
    \begin{equation} \label{equation:lower-bound-relation}
        a_{n+1}\ \ge\ a_n + a_{n-k}.
    \end{equation}
\end{proposition}

\begin{proof}
    Since $a_{n+1}$ is the smallest number that does not have a $S$-LID decomposition using the first $n$ terms, it suffices to show that every $r < a_{n} + a_{n - k}$ has an $S$-LID decomposition using the first $n$ terms.

    If $0 \le r < a_{n}$, then, by definition of $a_{n}$, there exists an $S$-LID decomposition of $r$ using the first $n-1$ terms; therefore, the same decomposition is also an $S$-LID decomposition of $r$ using the first $n$ terms.
  
    If  $a_{n} \le r < a_{n} + a_{n-k}$, then $0 \le r - a_{n} < a_{n-k}$.
    By definition of $a_{n-k}$, there exists an $S$-LID decomposition of $r - a_n$ using the first $n-k-1$ terms, i.e., there exists $L \subset \{1, \dots, n-k-1\}$ such that $|\ell_1 - \ell_2| \notin S$ for all $\ell_1, \ell_2 \in L$ and
    \begin{equation}
        r - a_{n-1}\ =\ \sum_{\ell \in L} a_\ell.
    \end{equation}

    Since $|n - \ell| > k$ and, consequently, $|n - \ell| \notin S$ for all $\ell \in L$, it follows that $L \cup \{n\} \subset \{1, \dots, n\}$ satisfies $|\ell_1 - \ell_2| \notin S$ for all $\ell_1, \ell_2 \in L \cup \{n\}$ and
    \begin{equation}
        r\ =\ \sum_{\ell \in L \cup \{n\}} a_\ell;
    \end{equation}
    therefore, $r$ has an $S$-LID decomposition using the first $n$ terms.
\end{proof}

We remark that, although the Fibonacci quilt and triangular quilt sequences are not $S$-LID sequences, the equality case of Inequation~\eqref{equation:lower-bound-relation}, explicitly
\begin{equation} \label{equation:lower-bound-equality}
    a_{n+1}\ =\ a_{n} + a_{n - k},
\end{equation}
is related to them as well. For the Fibonacci quilt sequence, the equality (with $k = 4$) holds for all $n \geq 6$ (refer to Theorem~\ref{theorem:fibonacci-quilt-recurrence}); and for the triangular quilt sequence, the equality (with $k = 5$) holds for some values of $n$ (refer to Proposition~\ref{proposition:triangular-quilt-recurrence}).

The problem that this article partially solves is to determine, given a set $S$, if the Equation~\eqref{equation:lower-bound-equality}, applied to the $S$-LID sequence, holds for all sufficiently large positive integers $n$.
In Section~\ref{section:examples}, we understand some examples of $S$-LID sequences and we state some propositions similar to Proposition~\ref{proposition:triangular-quilt-recurrence} that show that this problem is interesting.
In Section~\ref{section:recurrence}, we display an infinite family of sets $S$ for which Equation~\eqref{equation:lower-bound-equality} holds for all sufficiently large positive integers $n$. More explicitly, we prove the following theorem:

\begin{theorem}\label{theorem:fixed_fringes_weak}
    Let $T$ be a finite set of positive integers. Let $c = \max{T}$. For all sufficiently large $k$, the $\big([k]\setminus (k-T)\big)$-LID sequence $(a_n)_{n=1}^\infty$ satisfies
    \begin{equation} \label{equation:general_recurrence}
        a_{n+1}\ =\ a_{n} + a_{n-k}.
    \end{equation}
    for all $n > k+c$.
\end{theorem}
Refer to Theorem~\ref{theorem:fixed_fringes} for the explicit meaning of ``sufficiently large $k$.''

\section{Examples of \texorpdfstring{$S$}{S}-legal index difference sequences} \label{section:examples}


The $S$-LID sequences are natural generalizations of the $2$-dimensional processes that yield the Fibonacci quilt and triangular quilt sequences.
Furthermore, this family of sequences also provides generalizations to important and standard sequences, such as the Fibonacci sequence and the sequence of powers of $2$.
In this section, we explore some examples of $S$-LID sequences and understand their structures.

\subsection{``Sparse'' sets \texorpdfstring{$S$}{S}} 
In this subsection, we discuss the $\varnothing$-LID, $\{1\}$-LID, $\{2\}$-LID, and $\{3\}$-LID sequences.

\begin{example}[The $\varnothing$-LID sequence]
    The $\varnothing$-LID sequence is the sequence of powers of $2$.

    Let $a_1, a_2, \dots$ be the $\varnothing$-LID sequence, i.e., the sequence of positive integers such that, for each positive integer $n$, $a_n$ is the smallest positive integer that cannot be written as a sum of distinct numbers in $\{a_1, a_2, \dots, a_{n-1}\}$.
    
    Note that the sequence $2^0, 2^1, \dots$ of powers of $2$ has the property that every number can be uniquely written as a sum of distinct terms. It follows that the sequence of powers of $2$ satisfies the definition of the $\varnothing$-LID sequence, thus $a_n = 2^{n-1}$ for all positive integers $n$.
\end{example}

\begin{example}[The $\{1\}$-LID sequence]
    The $\{1\}$-LID is the Fibonacci sequence.

    Let $a_1, a_2, \dots$ be the $\{1\}$-LID sequence, i.e., the sequence of positive integers such that, for each positive integer $n$, $a_n$ is the smallest positive integer that cannot be written as a sum of distinct numbers in $\{a_1, a_2, \dots, a_{n-1}\}$ without using two summands $a_i$ and $a_j$ such that $|i - j| = 1$.
    
    This definition is equivalent to the alternative definition of Fibonacci numbers presented in Definition~\ref{proposition:alternate-fibonacci}. Hence, $a_n = F_{n}$ for all positive integers $n$.
\end{example}

We remark that the $\{1\}$-LID sequence satisfies the equality case of Inequation~\eqref{equation:lower-bound-relation} for $k = \max \{1\} = 1$, explicitly
\begin{equation}
    a_{n+1}\ =\ a_{n} + a_{n - 1},
\end{equation}
for all integers $n \geq 2$.

\begin{proposition}[The $\{2\}$-LID sequence]
    Let $a_1, a_2, \dots$ be the $\{2\}$-LID sequence.
    Then, $a_1 = 1, a_2 = 2, a_3 = 4$, and
    \begin{gather}
        a_{n+1}\ =\ a_{n} + a_{n-2}, \text{ and} \label{equation:2-lid} \\
        a_{n+1} > \sum_{i = 1}^{n-2} a_i \label{equation:2-lid:aux}
    \end{gather}
    for all integers $n \geq 3$.
\end{proposition}

\begin{proof}
    We use induction on $n$.
    Note that $a_1 = 1$, $a_2 = 2$, $a_3 = 4$, $a_4 = 5$, $a_5 = 7$. Hence, the statement is true for $n \in \{3, 4\}$.
    
    Let $n \geq 5$.
    Assume that $a_{t+1} = a_t + a_{t-2}$ and $a_{t+1} > \sum_{i = 1}^{t-2} a_i$ for all $3 \le t < n$.
    From Proposition~\ref{proposition:lower-bound-relation}, it follows that 
    \begin{equation}  \label{equation:lower-bound-relation-2-LID}
    a_{n+1} \geq a_n + a_{n-2}.
    \end{equation}
    
    Suppose $a_n + a_{n-2}$ has a $\{2\}$-LID decomposition using the first $n$ terms of the $\{2\}$-LID sequence, i.e., there exists $L \subset [n]$ such that 
    \begin{equation}
        a_n + a_{n-2}\ =\ \sum_{\ell \in L} a_\ell
    \end{equation}
    and $|i - j| \neq 2$ for all $i, j \in L$.
    
    Suppose $n \in L$. Then,
    \begin{equation}
        a_{n-2}\ =\ \sum_{\ell \in L\setminus\{n\}} a_\ell.
    \end{equation}
    Note that $n-1 \notin L$, otherwise $a_{n-2} = \sum_{\ell \in L\setminus\{n\}} a_\ell \geq a_{n-1} > a_{n-2}$.
    Also note that $n-2 \notin L$, since $|n - (n-2)| = 2$. Hence, $L \subset [n-3]$.
    Therefore, $a_{n-2}$ has a $\{2\}$-LID decomposition using the first $n-3$ terms of the $S$-LID sequence, a contradiction.
    
    Suppose $n \notin L$ and $n-1 \in L$. Then,
    \begin{equation}
        a_{n-2} + a_{n-3}\ =\ a_n - a_{n-1} + a_{n-2} 
                         \ =\ \sum_{\ell \in L\setminus\{n-1\}} a_\ell.
    \end{equation}
    Since $\left|(n-1) - (n-3)\right|\ =\ 2$, it follows that $n-3 \notin L$.
    If $n-2 \in L$, then $a_{n-3}\ =\ \sum_{\ell \in L\setminus\{n-1, n-2\}} a_\ell$, hence $a_{n-3}$ has a $\{2\}$-LID decomposition using the first $n-4$ terms of the $S$-LID sequence, a contradiction.
    If $n-2 \notin L$, then 
    \begin{equation}
        \sum_{i = 1}^{n-4} a_i\ <\ a_{n-1}\ =\ a_{n-2} + a_{n-4}\ \le\ a_{n-2} + a_{n-3}\ =\ \sum_{\ell \in L\setminus\{n-1\}} a_\ell\ \le\ \sum_{i = 1}^{n-4} a_i,
    \end{equation}
    a contradiction.

    Suppose $n, n-1 \notin L$. Then,
    \begin{equation}
        \sum_{i = 1}^{n-3} a_i + a_{n-2}\ <\ a_{n} + a_{n-2}\ =\ \sum_{\ell \in L} a_\ell \ \le\ \sum_{i=1}^{n-2} a_i,
    \end{equation}
    a contradiction.
    
    Therefore, $a_n + a_{n-2}$ does not have a $\{2\}$-LID decomposition using the first $n$ terms of the $\{2\}$-LID sequence. Then, $a_{n+1} \le a_n + a_{n-2}$, and consequently, with Equation~\eqref{equation:lower-bound-relation-2-LID}, $a_{n+1} = a_n + a_{n-2}$.
    Moreover, $a_{n+1} = a_n + a_{n-2} > a_{n-2} + \sum_{i = 1}^{n-3} a_i = \sum_{i = 1}^{n-2} a_i$.
    
    Therefore, by induction, Equations~\eqref{equation:2-lid} and \eqref{equation:2-lid:aux} hold for all integers $n \geq 3$.
\end{proof}

We remark that the $\{2\}$-LID sequence is the OEIS sequence A164316 \cite{OEISA164316}.


For $k \in \{1, 2\}$, the $\{k\}$-LID sequence satisfies $a_{n+1} = a_{n} + a_{n-k}$ for all $n \geq k+1$.
However, this fact is not true for $k = 3$.

\begin{example}[The $\{3\}$-LID sequence]
    The first twenty terms of the $\{3\}$-LID sequence are 
    \begin{equation}
        1, 2, 4, 8, 9, 11, 15, 23, 32, 64, 79, 134, 166, 244, 355, 489, 679, 1011, 1485, 2163.
    \end{equation}
\end{example}

\begin{proposition} \label{proposition:computation:3-lid}
    Let $4 \leq n \leq 40$ be an integer.
    Let $a_1, a_2, \dots$ be the $\{3\}$-LID sequence.
    Then, 
    \begin{equation}
        a_{n+1}\ =\ a_{n} + a_{n-3}
    \end{equation}
    holds if, and only if, $n \in \{4, 5, 6, 7, 8, 10, 12, 15\}$.
\end{proposition}

Proposition~\ref{proposition:computation:3-lid} was verified using computer code.%
\footnote{Refer to the GitHub repository at \url{https://github.com/ZeusDM/S-LID-sequences} for the details of this verification.}

\subsection{``Dense'' sets \texorpdfstring{$S$}{S}}
In this subsection, we discuss the $S$-LID sequences for $S = [k] \setminus (k - T)$, for some relatively small sets $T$: $T = \varnothing$, $T = \{1\}$, and $T=\{2\}$.
In Section~\ref{section:recurrence}, we deal with arbitrary sets $T$ and sufficiently large $k$. 

\begin{lemma} \label{lemma:k-slid-decomposition}
    Let $k, n$ be integers.
    Let $a_1, a_2, \dots$ be the $[k]$-LID sequence.
    The largest integer with a $[k]$-LID decomposition using $\{a_1, \dots, a_{n}\}$ is
    \begin{equation}
        \sum_{i=0}^{\left\lfloor \frac{n+1}{k+1} \right\rfloor} a_{n-i(k+1)}.
    \end{equation}
    Moreover,
    \begin{equation}
        \begin{aligned}
            a_{n+1}\ \le\ 1 + \sum_{i=0}^{\left\lfloor \frac{n+1}{k+1} \right\rfloor} a_{n-i(k+1)}.
        \end{aligned}
    \end{equation}
\end{lemma}

\begin{proof}
    Suppose $m$ has a $[k]$-LID decomposition using $\{a_1, \dots, a_{n}\}$, i.e, there exists $L \subset [n]$ such that 
    \begin{equation}
        m\ =\ \sum_{\ell \in L} a_{\ell} 
    \end{equation}
    and $|i - j| \geq k+1$ for all distinct $i, j \in L$.
    Let $\ell_1 > \ell_2 > \ell_3 > \cdots < \ell_{|L|}$ be the elements of $L$.
    We know that $\ell_1 \le n$ and $\ell_t - \ell_{t+1} \geq k+1$ for all positive integers $t < |L|$.
    Hence, since the $S$-LID sequence is increasing, it follows that
    \begin{equation}
        m\ =\ \sum_{t=1}^{|L|} a_{\ell_t}\ \le\ \sum_{i=0}^{\left\lfloor \frac{n+1}{k+1} \right\rfloor} a_{n-i(k+1)}.
    \end{equation}
    Therefore, the first claim follows.
    Since $a_{n+1}$ is the smallest positive integer with no $[k]$-LID decomposition using $\{a_1, \dots, a_{n}\}$, the second claim follows.
\end{proof}

\begin{proposition} \label{proposition:k-slid}
    Let $k$ be an integer.
    Let $a_n$ denote the $n$\textsuperscript{th} term of the $S$-LID sequence.
    Then, for all nonnegative integers $n$,
    \begin{equation} \label{equation:1:k-slid}
        a_{n+1}\ =\ 1 + \sum_{i=0}^{\left\lfloor \frac{n+1}{k+1} \right\rfloor} a_{n-i(k+1)}
    \end{equation}
    and
    \begin{equation} \label{equation:2:k-slid}
        \begin{aligned}
            a_{n+1} \ &=\ n+1 & \text{ if } n \le k, \\
            a_{n+1} \ &=\ a_n + a_{n-k} & \text{ if } n \geq k+1.
        \end{aligned}
    \end{equation}
\end{proposition}

\begin{proof}
    We use induction on $n$.
    Recall that $a_1 = 1$.
    There are no $[k]$-LID decompositions with more than $2$ summands using the first $n$ terms of the $[k]$-LID sequence, for all nonnegative integers $n \le k$.
    Therefore, $a_{n+1} = a_{n} + 1$  for all nonnegative integers $n \le k$.
    By induction, it follows that $a_{n+1} = {n+1}$ for all nonnegative integers $n \le k$.

    Now, let $n \geq k + 1$.
    Assume, by induction hypothesis, that equation~\eqref{equation:1:k-slid} holds for $n-k-1$, i.e., 
    \begin{equation}
        a_{n-k}\ =\ 1 + \sum_{i=0}^{\left\lfloor \frac{n-k}{k+1} \right\rfloor} a_{n-k-1=i(k+1)}
                 =\ 1 + \sum_{i=1}^{\left\lfloor \frac{n+1}{k+1} \right\rfloor} a_{n-(i+1)(k+1)}.
    \end{equation}
    
    Lemma~\ref{lemma:k-slid-decomposition} implies that 
    \begin{equation} \label{equation:proposition:k-slid:upper-bound}
            a_{n+1}\ \le\ 1 + \sum_{i=0}^{\left\lfloor \frac{n+1}{k+1} \right\rfloor} a_{n-i(k+1)}
            =\ a_{n} + a_{n-k},
    \end{equation}
    and Proposition~\ref{proposition:lower-bound-relation} implies that
    \begin{equation}  \label{equation:lower-bound-relation-k-LID}
        a_{n+1}\ \geq\ a_n + a_{n-k}.
    \end{equation}

    Therefore, equality holds in Equations~\eqref{equation:proposition:k-slid:upper-bound}~and~\eqref{equation:lower-bound-relation-k-LID}, as desired for the induction.
\end{proof}

For sets $S$ of the form $[k] \setminus \{k - 1\}$ and $[k] \setminus \{k - 2\}$, we have Conjectures~\ref{conjecture:computation1} and \ref{conjecture:computation2}, based on computational experiments.%
\footnote{Refer to the GitHub repository at \url{https://github.com/ZeusDM/S-LID-sequences} for the details of these verifications.}

\begin{conjecture} \label{conjecture:computation1}
    Let $k \geq 2$. Let $S = [k] \setminus \{k-1\}$.
    Let $a_1, a_2, \dots$ be the $S$-LID sequence. Then,
    \begin{equation} \label{equation:computation1}
        a_{n+1}\ =\ a_{n} + a_{n-k}
    \end{equation}
    for all integers $n > k + 1$.
\end{conjecture}

\begin{conjecture} \label{conjecture:computation2}
    Let $k \geq 8$.
    Let $S=[k] \setminus \{k-2\}$.
    Let $a_1,a_2,\dots$ be the $S$-LID sequence.
    Then,
    \begin{equation} \label{equation:computation2}
        a_{n+1}\ =\ a_{n} + a_{n-k}
    \end{equation}
    for all integers $n > k + 2$.
\end{conjecture}

Equation~\eqref{equation:computation1} holds for all integers $k$ and $n$ satisfying $2 \le k \le 20$ and all $k + 1 \le n \le 40$.
Equation~\eqref{equation:computation2} holds for all integers $k$ and $n$ satisfying $2 \le k \le 20$ and all $k + 1 \le n \le 62$.



%
%

%


\section{Recurrence relations for \texorpdfstring{$S$}{S}-LID sequences} \label{section:recurrence}

In this section, we prove Theorem~\ref{theorem:fixed_fringes_weak}, which is a natural generalization of Proposition~\ref{proposition:k-slid}, Conjecture~\ref{conjecture:computation1} and Conjecture~\ref{conjecture:computation2}.
The version of Theorem~\ref{theorem:fixed_fringes_weak} we prove by the end of this section is 
Theorem~\ref{theorem:fixed_fringes}, in which we explicitly define what ``sufficiently large $k$'' means.

The proof is split into two parts.
In Subsection~\ref{subsection:finite-check}, we describe a condition that, if a set $S$ satisfies, then the Equation~\eqref{equation:general_recurrence} holds for the $S$-LID sequence for all suitable $n$.
In Subsection~\ref{subsection:fixed_fringes}, we show that for all finite sets $T$, for all sufficiently large $k$, the set $[k]\setminus (k-T)$ satisfies the condition from Subsection~\ref{subsection:finite-check}, hence Equation~\eqref{equation:general_recurrence} holds for the $\big([k]\setminus (k-T)\big)$-LID sequence for all suitable $n$.

\subsection{Finite check to prove recurrence} \label{subsection:finite-check}

We now prove some preliminary results about $S$-LID sequences that outline a finite check that can be used to show that certain $S$-LID sequences satisfy the recurrence relation given in Equation~\eqref{equation:general_recurrence}. Fix $(a_n)_{n=1}^\infty$ to be the $S$-LID sequence corresponding to a set $S$ with largest element $k$ such that the smallest positive integer not in $S$ is $k-c$ for some $c>0$. Note from Proposition~\ref{proposition:lower-bound-relation} that the recurrence always gives a lower bound for the next term in the sequence.

The proof of our finite check follows by simultaneous induction on the following statements, where $d$ is a positive integer:
\begin{align*}
    A(n):\quad&a_{n + 1}\ =\ a_n + a_{n-k}, \\
    B(n):\quad&a_n + a_{n-k}\ >\ \sum_{\mathclap{{0\le i(k-c)< n-1}}} a_{n-1-i(k-c)}\ =\ a_{n-1} + a_{n-1-(k-c)} + a_{n-1-2(k-c)} + \cdots, \\
    C_d(n):\quad&a_{n+1} + a_{n}\ >\ a_{n+c} + a_{n-d}. \\
\end{align*}

In particular, we will show that if $A(n)$, $B(n)$, and $C_d(n)$ all hold for certain intervals of base cases (for a certain value of $d$), then $A(n)$ holds for $n>k+c$. We first prove a series of lemmas showing how these statements relate to each other.

\begin{lemma}\label{lemma:B_implies_A}
    $B(n)$ implies $A(n)$.
\end{lemma}

\begin{proof}
     Assume $A(n)$ does not hold. By Proposition~\ref{proposition:lower-bound-relation},  $a_n+a_{n-k}$ has an $S$-LID decomposition using $(a_1,\ldots,a_{n})$, i.e., there exists $L\subset [n]$ such that $|\ell_1-\ell_2|\notin S$ for all $\ell_1,\ell_2\in L$ and
    \begin{equation}
        a_n + a_{n-k}\ =\ \sum_{\ell\in L} a_\ell.
    \end{equation}
    If $n\in L$, we would have
    \begin{equation}
        a_{n-k}\ =\ \sum_{\ell\in L\setminus\{n\}} a_\ell.
    \end{equation}
    If this was the case, $n-k\notin L$ because $|n-(n-k)| = k \in S$. Since $(a_i)_{i=1}^{\infty}$ is increasing, this would yield an $S$-LID decomposition of $a_{n-k}$ using $(a_1,\ldots,a_{n-(k+1)})$, which cannot occur. Therefore, $n\notin L$.
    
    Given that $n \notin L$, the value for $\sum_{\ell\in L} a_\ell$ is at most
    \begin{equation}
        \sum_{\mathclap{{0\le i(k-c)< n-1}}} a_{n-1-i(k-c)}\ =\ a_{n-1}+a_{n-1-(k-c)}+a_{n-1-2(k-c)}+\cdots.
    \end{equation}
    This is because because $k-c$ is the smallest positive integer not in $S$, and thus the smallest legal index gap in any $S$-LID decomposition. Because $(a_n)_{n=1}^{\infty}$ is increasing, beginning with $a_{n-1}$ and choosing the smallest gap for each term yields an upper bound for the value for this sum. Therefore,
    \begin{equation}
        a_n + a_{n-k}\ =\ \sum_{\ell\in L} a_\ell\ \le\ a_{n-1} + a_{n-1-(k-c)} + a_{n-1-2(k-c)} + \cdots,
    \end{equation}
    Therefore, $B(n)$ also does not hold. Thus, $B(n)$ implies $A(n)$.
\end{proof}

\begin{lemma}
\label{lemma:arith-series}
    Let $c'$ be a positive integer satisfying $k \geq 2(c'+1)$. Then, for all $n\ge 1$,
    \begin{equation}\label{equation:arith-series-lemma}
        a_n\ >\ \sum_{\mathclap{k-c'\le i(k-c')< n}} a_{n-i(k-c')} \ =\ a_{n-(k-c')}+a_{n-2(k-c')}+\cdots.
    \end{equation}
\end{lemma}

\begin{proof}
       Induct on $n$. For a base case, for all $1 \le n\le 2(k-c')$, the right hand side has at most one summand of index strictly smaller than $n$. Now take $n>2(k-c')$ and suppose as an induction hypothesis Equation~\eqref{equation:arith-series-lemma} for $n-2(k-c')$, that is
    \begin{equation}\label{equation:gap_sum_induct}
        a_{n-2(k-c')}\ >\ a_{n-3(k-c')} + a_{n-4(k-c')} + \cdots.
    \end{equation}
    Applying Proposition~\ref{proposition:lower-bound-relation} $k-c'$ times, we have
    \begin{equation} \label{equation:indexed-sum}
        \begin{aligned}
            a_n\ &\ge\ a_{n-1}+a_{n-1-k}\\
                 &\ge\ (a_{n-2}+a_{n-2-k})+a_{n-1-k}\\
                 &\ge\ ((a_{n-3}+a_{n-3-k})+a_{n-2-k})+a_{n-1-k}\\
                 &\,\,\,\vdots\\
                 &\ge\ a_{n-(k-c')}+\sum_{i=n-(k-c')-k}^{n-1-k}a_i,
        \end{aligned}
    \end{equation}
    where we say $a_i=0$ if $i\le 0$, which is relevant only for $n\le 2k-c'$. Because $k\ge 2(c'+1)$, we have $n-(k-c')-k<n-2(k-c')$ and $n-2(k-c')+1\le n-1-k$, so $a_{n-2(k-c')+1}$ and $a_{n-2(k-c')}$ are both in the summation in Equation~\eqref{equation:indexed-sum}. Removing all other summands but these two gives
    \begin{equation}
        a_n\ \ge\ a_{n-(k-c')}+a_{n-2(k-c')}+a_{n-2(k-c')+1}\ >\ a_{n-(k-c')}+a_{n-2(k-c')}+a_{n-2(k-c')}.
    \end{equation}
    Using Equation~\eqref{equation:gap_sum_induct} to rewrite the second $a_{n-2(k-c')}$ gives 
    \begin{equation}
        a_n\ >\ a_{n-(k-c')}+a_{n-2(k-c')}+ a_{n-3(k-c')} + a_{n-4(k-c')} + \cdots
    \end{equation}
    which completes the inductive proof.
\end{proof}

\begin{corollary}\label{cor:gap_sum_bound}
    For all $n\ge 1$,
    \begin{equation}
        a_{n+(k-c')}\ >\ \sum_{\mathclap{0\le i(k-c)<n}} a_{n-i(k-c)} \ =\ a_n + a_{n-(k-c)} + a_{n-2(k-c)} + \cdots
    \end{equation}
    for all $c'\ge c>0$ with $k\ge 2(c'+1)$.
\end{corollary}

\begin{proof}
   Because $c'\ge c$ and the sequence is increasing
    \begin{equation}
        a_\ell + a_{\ell-(k-c')} + a_{\ell-2(k-c')} + \cdots \ \ge\ a_\ell + a_{\ell-(k-c)} + a_{k-2(k-c)} + \cdots
    \end{equation}
    The result then follows from applying Lemma~\ref{lemma:arith-series} to the left hand side.
\end{proof}

\begin{lemma}\label{lemma:C_implies_B}
    $C_d(n-k-1)$ implies $B(n)$ for $k\ge 2d-4c+2$.
\end{lemma}

\begin{proof}
    By Proposition~\ref{proposition:lower-bound-relation}, we have $a_n-a_{n-1}\ge a_{n-(k+1)}$, so to prove $B(n)$ it suffices to show that
    \begin{equation}
        a_{n-k} + a_{n-(k+1)}\ >\ a_{n-1-(k-c)} + a_{n-1-2(k-c)} + \cdots.
    \end{equation}
    Assuming $C_d(n-k-1)$, we have
    \begin{equation}
        a_{n-k}+a_{n-k-1} - a_{n-k-1+c}\ >\ a_{n-k-1-d}
    \end{equation}
    so it suffices to show that
    \begin{equation}
        a_{n-k-1-d}\ >\ a_{n-1-2(k-c)}+a_{n-1-3(k-c)}+\cdots.
    \end{equation}
    By Corollary~\ref{cor:gap_sum_bound}, the right hand side is bounded above by
    \begin{equation}
        a_{n-1-2(k-c)+(k-c')}\ =\ a_{n-k-1-2c-c'}
    \end{equation}
    for all $c'\ge c$ and $k\ge 2(c'+1)$. We have $n-k-1-d \ge n-k-1-2c-c'$ if and only if $c' \ge d-2c$, so for $k\ge 2(d-2c+1) = 2d-4c+2$ we obtain the desired inequation.
\end{proof}

\begin{lemma}\label{lemma:C_induction}
    $C_d(n)$, $C_d(n-k)$, $A(n+1)$, $A(n)$, $A(n+c)$, and $A(n-d)$ together imply $C_d(n+1)$.
\end{lemma}

\begin{proof}
   $C_d(n)$ and $C_d(n-k)$ imply that
    \begin{align}
        a_{n+1} + a_{n}     \ &>\ a_{n+c} + a_{n-d},\\
        a_{n+1-k} + a_{n-k} \ &>\ a_{n+c-k} + a_{n-d-k}.
    \end{align}
    Adding these inequations together and applying the recursion relations $A(n+1)$, $A(n)$, $A(n+c)$, and $A(n-d)$ gives
    \begin{equation}
        a_{n+2}+a_{n+1}\ >\ a_{n+c+1} + a_{n-d+1},
    \end{equation}
    which is precisely $C_d(n+1)$.
\end{proof}

We now put these results together.

\begin{theorem}\label{theorem:finite_check}
    Let $S$ be a finite set of positive integers with largest element $k$, and let $k-c$ be the smallest positive integer not in $S$ (so $S$ is not $[k]$). Let $(a_n)_{n=1}^\infty$ be the $S$-LID sequence. Suppose that there exists $d>0$ such that
    \begin{itemize}
        \item $k\ge 2d-4c+2$, 
        \item $C_d(n)$ holds for $c+d+1\le n\le k+c+1+d$, and
        \item $B(n)$ holds for $k+c+1\le n\le 2k+c+d+2$. 
    \end{itemize}
    Then
    \begin{equation}\label{equation:finite_check_recursion}
        a_{n+1}\ =\ a_{n} + a_{n-k}
    \end{equation}
    for all integers $n > k + c$.
\end{theorem}

\begin{proof}
    Suppose first that $k+c+1\le n\le 2k+c+d+2$. Then by Lemma~\ref{lemma:B_implies_A}, $A(n)$ holds because $B(n)$ holds. Now take $n>2k+c+d+2$, and suppose for induction that $A(m)$ and $B(m)$ hold for $k+c+1\le m<n$ and $C_d(m-k-1)$ holds for $k+c+d+2\le m < n$.
    
    Note that, because $n\ge 2k+c+d+3$ and $k > c$,
    \begin{equation}
        k+c=1\ \le\ n-k-2-d\ \le\ n-k-2\ <\ n-k-1\ \le\ n-k-2+c\ <\ n,
    \end{equation}
    therefore $A(n-k-2)$, $A(n-k-1)$, $A(n-k-2+c)$, and $A(n-k-2-d)$ hold.
    Moreover, $C_d(n-k-2)$ and $C_d(n-2k-2)$ both hold because $n-2k-2 \ge c+d+1$, and $n-k-2<n-k-1$. Therefore, $C_d(n-k-1)$ also holds by Lemma~\ref{lemma:C_induction}.
    
    By Lemma~\ref{lemma:C_implies_B}, we then have $B(n)$ as $k\ge 2d-4c+2$, which then implies $A(n)$ by Lemma~\ref{lemma:B_implies_A}. We have thus increased by $1$ the values of $n$ such that $A(n)$, $B(n)$, and $C_d(n-k-1)$ all hold. By induction, we conclude that $A(n)$ and $B(n)$ hold for all $n\ge k+c+1$ and $C_d(n-k-1)$ holds for all $n\ge k+c+d+2$. In particular, Equation~\eqref{equation:finite_check_recursion} holds for all $n>k+c$.
\end{proof}

Given any fixed $S$, there are only finitely many $d>0$ such that $k\ge 2d-4c+2$, so this theorem gives a finite check to show that any $S$-LID sequence satisfies the expected recurrence relation. Although these conditions need not occur in every sequence satisfying the recurrence, the next section gives many families of $S$-LID sequences for which this check is sufficient to prove the result.

\subsection{Removing fixed fringes} \label{subsection:fixed_fringes}

In this section we use Theorem~\ref{theorem:finite_check} to prove recurrence relations for $S$-LID sequences $(a_n)_{n=1}^\infty$ when $S = [k]\setminus(k-T)$, where $T$ is a fixed non-empty set of positive integers with largest element $c$ and $k\gg 0$.

\begin{lemma}\label{lemma:initial_values}
    The $S$-LID sequence satisfies $a_n = n$ for $1\le n\le k-c+1$. Moreover, for any integer $i>-k$ there exists an integer $f_T(i)$ such that $a_{k+i} = k + f_T(i)$ for all $k\ge i+2c-1$.
\end{lemma}

\begin{proof}
    The first claim follows from the fact that there are no $S$-LID decompositions using more than one summand from the collection $(a_1,\ldots,a_{k-c})$. This immediately implies the second claim with $f_T(i) = i$ for $-k< i\le -c+1$. Note that the only $S$-LID decomposition using more than one summand using the collection $(a_1,\ldots,a_{k-c+1})$ is $a_1+a_{k-c+1} = k-c$.
    
    Fix $i>-c+1$, and suppose for induction that we have the desired constant $f_T(i-1)$ holding for $k\ge (i-1)+2c-1$. Suppose moreover that, for all such $k$, every element of the set of $S$-LID decompositions with more than one summand using the collection $(a_1,\ldots,a_{k+i-2})$ is of the form $k+(\text{constant independent of $k$})$.  Finally, fix $k\ge i+2c-1$, which in particular implies that we are in a case where the inductive hypothesis applies.  Then we would like to show that every element of the set of $S$-LID decompositions with more than one summand using the collection $(a_1,\ldots,a_{k+i-1})$ is also of the form $k+(\text{constant independent of $k$})$, which will immediately imply that $f_T(i)\coloneqq a_{k+i}-k$ is independent of $k$.
    
    By the inductive hypothesis, we only need to consider the $S$-LID decompositions that contain $a_{k+i-1}$. The decompositions with exactly two summands are of the form
    \begin{equation}
        a_{k+i-1}+a_{k+i-1-(k+r)}\ =\ k+f_T(i-1)+a_{i-1-r}
    \end{equation}
    for $-r\in T$ or $r>0$ such that $i-1-r>0$. Note that the possible choices for $r$ are independent of $k$: either $-r\in T$ or $1\le r<i-1$. Moreover, the values of each $a_{i-1-r}$ are also independent of $k$: we have $a_{i-1-r}=i-1-r$ for all $r$ we consider because the largest possible value for $i-1-r$ is
    \begin{equation}
        i-1+c\ \le\ (k-2c+1)-1+c\ <\ k-c+1.     
    \end{equation}

    Finally, note that the closest summand that can be added to $a_{k+i-1}$ in an $S$-LID decomposition is $a_{k+i-1-(k-c)} = a_{i-1+c}$, to which nothing smaller can legally be added as $i-1+c-(k-c)\le 0$. Hence, we do not need to consider $S$-LID decompositions with more than $2$ summands, and we are done by induction.
\end{proof}

\begin{lemma}\label{lemma:base_case_A_B}
    For any $d>0$, $A(n)$ and $B(n)$ are true for $k+c+1\le n\le 2k+c+d+2$ for all $k\ge 4c+d+1$ such that
    \begin{equation}
        k\ >\ \max_{-2c+2\le i\le c+d+2}\{f_T(i-1+c)-f_T(i-1)-f_T(i)+i+2c-1\}.
    \end{equation}
\end{lemma}

\begin{proof}
    Fix a value of $k$ as in the lemma statement. By Lemma~\ref{lemma:B_implies_A} it suffices to just prove $B(n)$. Suppose first that $k+c+1\le n\le 2k-2c+1$. The upper bound on $n$ implies that $n-1-2(k-c)\le 0$, so the only terms in the desired inequation are
    \begin{equation}
        a_n + a_{n-k}\ >\ a_{n-1}+ a_{n-1-(k-c)}.
    \end{equation}
    By Proposition~\ref{proposition:lower-bound-relation}, we have $a_n-a_{n-1}\ge a_{n-k-1}$, so it suffices to show that
    \begin{equation}
        a_{n-k-1}+a_{n-k}\ >\ a_{n-1-(k-c)}.
    \end{equation}
    The largest index represented in this inequation is $n-1-(k-c)$, and as $n<2k-2c+2$ we have $n-1-(k-c) < k-c+1$. Thus, by Lemma~\ref{lemma:initial_values}, each of the above $a_i$ can be replaced by $i$, and it suffices to prove that
    \begin{equation}
        (n-k-1)+(n-k)\ >\ n-1-(k-c),
    \end{equation}
    which is true for $n>k+c$.
    
    Now suppose that $2k-2c+2\le n \le 2k+c+d+2$. Note that the $a_{n-1-3(k-c)}$ term and beyond vanish in the desired inequation. This is because $n-1-3(k-c)\le 0$ holds for $n=2k+c+d+2$, which follows from the fact that $k\ge 4c+d+1$. Hence, the desired inequation is
    \begin{equation}
        a_n + a_{n-k}\ >\ a_{n-1}+a_{n-1-(k-c)}+a_{n-1-2(k-c)}.
    \end{equation}
    By Proposition~\ref{proposition:lower-bound-relation}, $a_n-a_{n-1}\ge a_{n-k-1}$, so it suffices to show that
    \begin{equation}
        a_{n-k-1}+a_{n-k}\ >\ a_{n-1-(k-c)}+a_{n-1-2(k-c)}.
    \end{equation}
    Writing $n = 2k+i$ for $-2c+2\le i\le c+d+2$, this is equivalent to showing that
    \begin{equation}
        a_{k+i-1}+a_{k+i}\ >\ a_{k+i-1+c}+a_{i+2c-1}.
    \end{equation}
    The maximal index that will occur in all of these inequations is $k+2c+d+1$, so from Lemma~\ref{lemma:initial_values}, we can replace each $a_{k+j}$ with $k + f_T(j)$ because $k>(2c+d+1)+2c-1 = 4c+d$. The desired inequation becomes
    \begin{equation}
        (k+f_T(i-1)) + (k+f_T(i))\ >\ (k+f_T(i-1+c)) + a_{i+2c-1},
    \end{equation}
    where $a_{i+2c-1} = i+2c-1$ as $i+2c-1\le (c+d+2)+2c-1\le k-c+1$ for $k\ge 4c+d$. Hence, the result holds for
    \begin{equation}
        k\ >\ f_T(i-1+c)-f_T(i-1)-f_T(i)+i+2c-1
    \end{equation}
    which is true for our choice of $k$.
\end{proof}

\begin{lemma}\label{lemma:base_case_C}
    Let $m$ be the smallest nonnegative integer such that for all $i\ge m$,
    \begin{equation}\label{equation:quadratic_bound}
        \frac{1}{2}i^2+\left(\frac{3}{2}-c\right)i+\left(1+\frac{1}{2}c-\frac{3}{2}c^2\right)\ >\ 0,
    \end{equation}
    and let $d$ be the maximum of $2c+m-1$ and
    \begin{equation}
        \max_{-2c+2\le i\le c+m}\{f_T(i+c)-f_T(i)-f_T(i+1)+i+1\}.
    \end{equation}
    Then $C_d(n)$ is true for $c+d+1\le n\le k+c+1+d$ for all $k\ge 4c+d+1$ such that
    \begin{equation}
        k\ >\ \max_{-2c+2\le i\le c+d+2}\{f_T(i-1+c)-f_T(i-1)-f_T(i)+i+2c-1\}.
    \end{equation}
\end{lemma}

\begin{proof}
    Fix such an $m$, $d$, and $k$. Suppose first that $c+d+1\le n\le k-2c+1$. Then the maximal possible index appearing in $C_d(n)$ is $(k-2c+1)+c = k-c+1$, so by Lemma~\ref{lemma:initial_values} we can replace the sequence terms in $C_d(n)$ with their indices. The desired condition then becomes
    \begin{equation}
        (n+1) + n\ >\ (n+c) + (n-d),
    \end{equation}
    which is true because $d\ge 2c+m-1\ge c$.
    
    Next, suppose that $k-2c+2\le n \le k+c+m$. Write $n = k+i$ for $-2c+2\le i \le c+m$. Because $d\ge 2c+m-1$, we have $k+i-d \le k+(c+m)-(2c+m-1) = k-c+1$, so $a_{k+i-d} = k+i-d$. Moreover, the maximal possible index appearing in $C_d(n)$ is $k+2c+m$, so because $k\ge 4c+d+1 \ge (2c+m)+2c-1$, we can apply Lemma~\ref{lemma:initial_values} to write $C_d(n)$ as
    \begin{equation}
        (k+f_T(i+1)) + (k+f_T(i))\ >\ (k+f_T(i+c)) + (k+i-d).
    \end{equation}
    This holds because
    \begin{equation}
        d\ >\ f_T(i+c) - f_T(i) - f_T(i+1) + i.
    \end{equation}
    
    Finally, suppose that $k+c+1+m\le n\le k+c+1+d$. Note that we have taken $k$ to be sufficiently large to apply Lemma~\ref{lemma:base_case_A_B}. For all $1\le t\le c+d$, we therefore have
    \begin{equation}
        a_{k+c+1+t}\ =\ a_{k+c+t}+a_{c+t},
    \end{equation}
    where $k+c+1\le k+c+t\le k+2c+d\le 2k+c+d+2$ because $k>c$, so we are indeed in the cases proven in Lemma~\ref{lemma:base_case_A_B}. Because, $k\ge 3c = (c+1)+2c-1$, we may take $a_{k+c+1} = k+f_T(c+1)$ by Lemma~\ref{lemma:initial_values}. By the same lemma, because $2c+d\le k-c+1$, we may replace any occurrence of $a_{c+t}$ with $c+t$.  It then follows from repeated application of the above recursion that
    \begin{align}
        a_{k+c+1+t}\ 
            &=\ a_{k+c+1} + a_{c+1} + a_{c+2} + \cdots + a_{c+t} \\
            &=\ k + f_T(c+1) + (c+1) + (c+2) + \cdots + (c+t) \\
            &=\ k + f_T(c+1) + \frac{1}{2}t^2 + \left(\frac{1}{2}+c\right)t, \label{equation:quadratic}
    \end{align}
    where this clearly extends to $t=0$ as well.
    
    Write $n = k+c+1+i$ for $m\le i\le d$. To prove $C_d(k+c+1+i)$, we want to show that
    \begin{equation}
        a_{k+c+2+i} + a_{k+c+1+i}\ >\ a_{k+2c+1+i} + a_{k+c+1+i-d}.
    \end{equation}
    We can write $a_{k+c+1+i-d} = (k+f_T(c+1+i-d))$ because $k\ge 3c-1 \ge (c+1+i-d)+2c-1$. Applying Equation~\eqref{equation:quadratic} to the remaining terms, the desired inequation becomes
    \begin{equation}
        \begin{gathered}
            \left(k + f_T(c+1) + \tfrac{1}{2}(i+1)^2 + \left(\tfrac{1}{2}+c\right)(i+1)\right)+\left(k + f_T(c+1) + \tfrac{1}{2}i^2 + \left(\tfrac{1}{2}+c\right)i\right)\\
           >\ \left(k + f_T(c+1) + \tfrac{1}{2}(i+c)^2 + \left(\tfrac{1}{2}+c\right)(i+c)\right)+(k+f_T(c+1+i-d)).
        \end{gathered}
    \end{equation}
    Canceling and rearranging, we obtain
    \begin{equation}
        \frac{1}{2}i^2+\left(\frac{3}{2}-c\right)i+\left(1+\frac{1}{2}c-\frac{3}{2}c^2\right)\ >\ f_T(c+1+i-d) - f_T(c+1).
    \end{equation}
    Note that as $i\le d$, we have $c+1+i-d\le c+1$. Therefore, because the $f_T(j)$ increase as $j$ increases, it suffices to prove that
    \begin{equation}
        \frac{1}{2}i^2+\left(\frac{3}{2}-c\right)i+\left(1+\frac{1}{2}c-\frac{3}{2}c^2\right)\ >\ 0.
    \end{equation}
    which is true for $i\ge m$ by the definition of $m$.
\end{proof}

We have thus established the base cases for this family of $S$-LID sequences that can be used to apply Theorem~\ref{theorem:finite_check}.

\begin{theorem}\label{theorem:fixed_fringes}
    Let $T$ be a non-empty finite set of positive integers with largest element $c$, define $d$ as in Lemma~\ref{lemma:base_case_C}, and suppose that $k\ge \max\{2d-4c+2,4c+d+1\}$ such that 
    \begin{equation}
        k\ >\ \max_{-2c+2\le i\le c+d+2}\{f_T(i-1+c)-f_T(i-1)-f_T(i)+i+2c-1\}.
    \end{equation}
    Then for $S = [k]\setminus (k-T)$, the $S$-LID sequence $(a_n)_{n=1}^\infty$ satisfies the recurrence
    \begin{equation}
        a_{n+1}\ =\ a_{n} + a_{n-k}
    \end{equation}
    for $n>k+c$.
\end{theorem}

\begin{proof}
    The criteria on $d$ and $k$ guarantee that we are in a situation where Lemmas~\ref{lemma:base_case_A_B} and \ref{lemma:base_case_C} can apply. These lemmas give the base cases for $B(n)$ and $C_d(n)$ respectively in Theorem~\ref{theorem:finite_check}, and taken with the fact that $k\ge 2d-4c+2$ they allow us to apply that theorem to obtain the desired recurrence.
\end{proof}

We note that lower bound on $k$ is only a lower bound needed to prove the recurrence by this general method, and need not be the smallest $k$ for which the recurrence holds for a given $T$. We demonstrate this by the following example:

\begin{example} \label{example:T=1}
    Take $T=\{1\}$, so that $S = [k]\setminus(k-T) = [k]\setminus\{k-1\}$ as in Conjecture~\ref{conjecture:computation1}. Then $c=1$, so Equation~\eqref{equation:quadratic_bound} becomes
    \begin{equation}
        \frac{1}{2}i^2 + \frac{1}{2}i\ =\ 0
    \end{equation}
    which implies that $m = 1$. From Lemma~\ref{lemma:initial_values} we have $f_T(i) = i$ for $-k<i\le 0$. By computing all $S$-legal sums using the initial values of the sequence, we obtain the values of $f_T(i)$ for some small values of $i$:
    \begin{equation}
        f_T(1)\ =\ 2, \quad 
        f_T(2)\ =\ 3, \quad 
        f_T(3)\ =\ 5, \quad 
        f_T(4)\ =\ 8.
    \end{equation}
    We want $d$ to be the maximum of $2c+m-1 = 2$ and
    \begin{equation}
        \begin{aligned}
            & \max_{-2c+2\le i\le c+m}\{f_T(i+c)-f_T(i)-f_T(i+1)+i+1\} \\
            =\ & \max_{0\le i\le 2}\{-f_T(i)+i+1\} = \max\{1,0,0\} = 1,
        \end{aligned}
    \end{equation}
    so we set $d = 2$. Finally, to apply Theorem~\ref{theorem:fixed_fringes}, we need $k$ to be at least $2d-4c+2 = 2$, at least $4c+d+1 = 7$, and larger than
    \begin{equation}
        \begin{aligned}
            &\max_{-2c+2\le i\le c+d+2}\{f_T(i-1+c)-f_T(i-1)-f_T(i)+i+2c-1\} \\
            =&\ \max_{0\le i\le 5}\{f_T(i-1)+i-2c-1\}\\
            =&\ \max\{-4,-2,1,3,6,10\}\ =\ 10.
        \end{aligned}
    \end{equation}
    Thus, for $T=\{1\}$, we need $k \geq 11$ for Theorem~\ref{theorem:fixed_fringes} to apply.
    Therefore, Conjecture~\ref{conjecture:computation1} does not directly follows from  Theorem~\ref{theorem:fixed_fringes}.
\end{example}

\section{Future work}

There are many questions that can now be explored for $S$-LID sequences, even those that eventually follow a simple recurrence. Below are just a few such questions that we leave for future work:

\begin{question}
    Is there a general proof of the recurrence relation in Theorem~\ref{theorem:fixed_fringes} that works for the minimal value of $k$ for which the recurrence holds?
\end{question}

\begin{question}
    Are there infinitely many indices of the triangular quilt sequence for which Equation~\eqref{equation:recurrence-triangular} holds?
    If yes, does the natural density of the set of these indices exist?
    If it exists, what is it?
    What about the analogous problem for $S$-LID sequences?
\end{question}

\begin{question}
    The algorithm to generate terms in $S$-LID sequences%
    \footnote{Refer to the GitHub repository at \url{https://github.com/ZeusDM/S-LID-sequences} for the algorithm we used to generate these terms.} is very slow.
    The time it takes to compute the $n$\textsuperscript{th} term given the previous terms grows exponentially on $n$.
    For the choices of $S$ covered in Section~\ref{section:recurrence}, using the recurrence relation gives us a very fast algorithm: given the previous terms, we can compute the next term in constant time.
    For the choices of $S$ not covered in Section~\ref{section:recurrence}, can one construct an algorithm better than the exponential one?
\end{question}

\begin{question}
    The greedy algorithm always returns the Zeckendorf decomposition of a number $n$; that's not true for Fibonacci quilt decompositions. In \cite[Theorem 1.13]{LegalDecompositionsNPLRS} the proportion of integers $[1,q_n)$ for which the greedy algorithm gives a legal decomposition is shown to converge to some computable constant. Does an analogous result hold for $S$-LID sequences?
\end{question}

\begin{question} \label{questin:numb_decomps}
    There is a unique legal decomposition of a positive integer into Fibonacci numbers, while in \cite[Theorem 1.11]{LegalDecompositionsNPLRS}, it is shown that the average number of ways to legally decompose an integer in $[0,q_{n})$ into terms in the Fibonacci quilt sequence grows exponentially in $n$. For various choices of $S$, how does the average number of $S$-LID decompositions of integers in $[0,a_{n})$ change with $n$?
\end{question}

As a first step toward answering Question~\ref{questin:numb_decomps}, we have the following lemma, which is an adaptation of \cite[Lemma 3.1]{LegalDecompositionsNPLRS}.
\begin{lemma}\label{lem:numb_decomps_recurr}
    Let $S = [k]\setminus \{k-c\}$ with $2c<k$, and $(a_n)_{m=1}^\infty$ be the $S$-LID sequence. Let $d_n$ denote the number of $S$-LID decompositions that can be formed with $\{a_1,\ldots,a_n\}$. Then the $d_n$ follow the recurrence relation
    \begin{equation}
        d_n\ =\ d_{n-1} + d_{n-k+c} - d_{n-k+c-1} + d_{n-k-1}
    \end{equation}
    for $n>k+1$.
\end{lemma}

\begin{proof}
    Let $c_n$ denote the number of $S$-LID decompositions that can be formed with $\{a_1,\ldots,a_n\}$ that include $a_n$. Then
    \begin{equation}\label{equation:c_n_d_n_relation1}
        c_n\ =\ d_n - d_{n-1}
    \end{equation}
    because the only decompositions that are counted by $d_n$ that are not counted by $d_{n-1}$ are those that use $a_n$. On the other hand, we have
    \begin{equation}\label{equation:c_n_d_n_relation2}
        c_n\ =\ d_{n-k-1} + c_{n-k+c}
    \end{equation}
    because the $S$-LID decompositions including $a_n$ partition into those whose next largest summand is at most $a_{n-k-1}$, which is counted by $d_{n-k-1}$, and those that include $a_{n-k+c}$, which is counted by $c_{n-k+c}$. Note that we use the fact that $2c<k$ in this second case to conclude that the next largest legal summand after $a_{n-k+c}$ is $a_{n-2k+2c}$, which has a strictly smaller index than $a_{n-k}$ and hence can be included regardless of the fact that we have already included $a_n$. Using Equation~\eqref{equation:c_n_d_n_relation1} to eliminate $c_n$ and $c_{n-k+c}$ from Equation~\eqref{equation:c_n_d_n_relation2} gives the lemma.
\end{proof}
We now outline how this lemma could be used to find the average number of $S$-LID decompositions. By Theorem~\ref{theorem:fixed_fringes}, for $k\gg 0$, we have $a_{n+1}-a_{n}-a_{n-k}=0$ for $n\gg 0$. Thus, for large $n$ we have $a_n = \sum_{i=1}^{k+1} \alpha_i\lambda_i^n$, where the $\alpha_i$ are constants and the $\lambda_i$ are the complex roots of $x^{k+1}-x^k-1$ (which are all distinct). Assuming one of the $\lambda_i$-terms dominates, say $\alpha_1\lambda_1^n$, which will be true if $|\lambda_1|>|\lambda_i|$ for all other $i$ and $\alpha_1\neq 0$, then we get the exponential approximation
\begin{equation}
    a_n\ =\ \alpha\lambda_1^n(1+o(1))
\end{equation}
for some constant $\alpha$. Similarly, from Lemma~\ref{lem:numb_decomps_recurr}, we have $d_n-d_{n-1}-d_{n-k+c}+d_{n-k+c-1}-d_{n-k-1}=0$. If the roots of $x^{k+1}-x^{k}-x^{c+1}+x^{c}-1$ are all distinct, and assuming some root $r$ dominates the behavior of the $d_n$, then we similarly obtain the exponential approximation
\begin{equation}
    d_n\ =\ \beta r^n(1+o(1))
\end{equation}
for some constant $\beta$. In a case where both exponential equations hold, note that $r\ge \lambda$ because $d_n\ge a_n$ always. When $r=\lambda$, we have that the average number of $S$-LID decompositions for integers in $[1,a_n)$ converges to a constant as $n\to\infty$, while if $r>\lambda$ this average grows like some multiple of $(r/\lambda)^n$ (this follows from an argument analogous to the proof of \cite[Theorem 1.11]{LegalDecompositionsNPLRS}). Note that all of the assumptions made can easily be checked by computer for any fixed choice of $S = [k]\setminus \{k-c\}$, although the problem of classifying all sets $S$ of this form will be left for future work.




\begin{thebibliography}{99}
    \bibitem{ZeckendorfsTheorem}
    \'{E}douard Zeckendorf,
    ``Repr\'{e}sentation des nombres naturels par une somme de nombres de {F}ibonacci ou de nombres de {L}ucas''.
    In: \emph{Bull. Soc. Roy. Sci. Li\`ege} 41 (1972), pp.~179--182.
    \textsc{issn:}~0037-9565.
    
    \bibitem{LegalDecompositionsNPLRS}
    Minerva Catral, Pari L. Ford, Pamela E. Harris, Steven J. Miller, and Dawn Nelson.
    ``Legal decomposition arising from non-positive linear recurrences''.
    In: \emph{Fibonacci Quart.} 54.4 (2016), pp.~348--365.
    \textsc{issn:}~0015-0517.
    \textsc{url:} \url{https://www.fq.math.ca/Papers1/54-4/CatFrdHarMilNel10202016.pdf}.
    \textsc{arXiv:} \url{https://arxiv.org/abs/1606.09312}.

    \bibitem{Ostrowski}
    Alexander Ostrowski.
    ``Bemerkungen zur Theorie der Diophantischen Approximationen''.
    In: \emph{Abh. Math. Sem. Univ. Hamburg} 1.1 (1922), pp.~77--98.
    \textsc{issn:}~1865-8784.
    \textsc{url:} \url{https://doi.org/10.1007/BF02940581}.

    \bibitem{Daykin}
    David E. Daykin.
    ``Representation of natural numbers as sums of generalized Fibonacci numbers. II''.
    In: \emph{Fibonacci Quart.} 7.5 (1969), pp. 494--510.
    \textsc{issn:}~0015-0517.
    \textsc{url:} \url{https://www.fq.math.ca/Scanned/7-5/daykin2.pdf}.

    \bibitem{f-decompositions}
    Philippe Demontigny, Thao Do, Archit Kulkarni, Steven J. Miller, David Moon, and Umang Varma.
    ``Generalizing Zeckendorf’s Theorem to $f$-decompositions''.
    In: \emph{J. Number Theory} 141 (2014), pp.~136--158.
    \textsc{issn:}~0022-314X.
    \textsc{url:} \url{https://doi.org/10.1016/j.jnt.2014.01.028}.
    \textsc{arXiv:} \url{https://arxiv.org/abs/1309.5599}.

    \bibitem{Alpert}
    Hannah Alpert.
    ``Differences of multiple Fibonacci numbers''.
    In: \emph{Integers} 9 (2009), A57, 745--749.
    \textsc{issn:}~1553-1732.
    \textsc{url:} \url{https://www.emis.de/journals/INTEGERS/papers/j57/j57.pdf}.

    \bibitem{Grabner}
    P. J. Grabner, R. F. Tichy, I. Nemes, and A. Peth\H{o}.
    ``Generalized Zeckendorf expansions''.
    In: \emph{Appl. Math. Lett.} 7.2 (1994), pp.~25--28.
    \textsc{issn:}~0893-9659.
    \textsc{url:} \url{https://doi.org/10.1016/0893-9659(94)90025-6}.
    
    \bibitem{Filipponi}
    P. Filipponi, P. J. Grabner, I. Nemes, A. Peth\H{o}, and R. F. Tichy.
    ``Corrigendum to: ``Generalized Zeckendorf expansions'' by Grabner, Tichy, Nemes and Peth\H{o}''.
    In: \emph{Appl. Math. Lett.} 7.6 (1994), pp.~25--26.
    \textsc{issn:}~0893-9659.
    \textsc{url:} \url{https://doi.org/10.1016/0893-9659(94)90087-6}.

    \bibitem{Timothy}
    Timothy J. Keller.
    ``Generalizations of Zeckendorf’s theorem''.
    In: \emph{Fibonacci Quart.} 10.1 (1972), pp.~95--102, 111, 112.
    \textsc{issn:}~0015-0517.
    \textsc{url:} \url{https://www.fq.math.ca/10-1.html}.

    \bibitem{Hamlin}
    Nathan Hamlin and William A. Webb.
    ``Representing positive integers as a sum of linear recurrence sequences''.
    In: \emph{Fibonacci Quart.} 50.2 (2012), pp.~99--105.
    \textsc{issn:} 0015-0517.
    \textsc{url:} \url{https://www.fq.math.ca/Papers1/50-2/HamlinWebb.pdf}.

    \bibitem{OEISA000931}
    OEIS Foundation Inc.~(2021), The On-Line Encyclopedia of Integer Sequences, \url{https://oeis.org/A000931}.

    \bibitem{OEISA164316}
    OEIS Foundation Inc.~(2021), The On-Line Encyclopedia of Integer Sequences, \url{https://oeis.org/A164316}.
\end{thebibliography}
\end{document}